\documentclass[reqno, 12pt]{amsart}

\usepackage{amssymb}
\usepackage{color}
\usepackage{verbatim}
\usepackage[dvips]{hyperref}

\newtheorem{theorem}{Theorem}[section]
\newtheorem{corollary}[theorem]{Corollary}
\newtheorem{lemma}[theorem]{Lemma}

\theoremstyle{definition}
\newtheorem{remark}[theorem]{Remark}

\theoremstyle{definition}

\theoremstyle{definition}
\newtheorem{assumption}[theorem]{Assumption}

\makeatletter
\def\dashint{\operatorname%
{\,\,\text{\bf--}\kern-.98em\DOTSI\intop\ilimits@\!\!}}
\makeatother

\newcommand{\WO}[2]{\overset{\scriptscriptstyle0}{W}\,\!^{#1}_{#2}}

\def\sfa{{\sf a}}
\def\sfb{{\sf b}}

\def\bR{\mathbb{R}}

\def\bQ{\mathbb{Q}}

\def\fL{\mathfrak{L}}

\def\cA{\mathcal{A}}
\def\cB{\mathcal{B}}
\def\cC{\mathcal{C}}
\def\cD{\mathcal{D}}

\def\cO{\mathcal{O}}

\title[Parabolic equations with partially VMO coefficients]
{Parabolic equations with partially VMO coefficients
and boundary value problems in Sobolev spaces with mixed norms}
\author{Doyoon Kim}



\begin{document}

\begin{abstract}
Second order parabolic equations in Sobolev spaces with mixed norms are studied.
The leading coefficients (except $a^{11}$) are measurable in both time and one spatial variable, and VMO in the other spatial variables.
The coefficient $a^{11}$ is measurable in time and VMO in the spatial variables.
The unique solvability of equations in the whole space is applied to solving Dirichlet and oblique derivative problems for parabolic equations defined in a half-space.
\end{abstract}

\maketitle

\section{Introduction}
In this paper we consider parabolic equations of the form
\begin{equation} \label{para_main_eq}
u_t + a^{ij}(t,x) u_{x^i x^j} + b^{i}(t,x) u_{x^i} + c(t,x) u = f
\end{equation}
in $L_{q,p}((S,T) \times \Omega)$,
$-\infty \le S < T \le \infty$,
where $\Omega$ is either $\bR^d$ or
$$
\bR^d_+ = \{ (x^1,x') \in \bR^d: x^1 > 0 \}.
$$
By $L_{q,p}((S,T) \times \Omega)$ we mean the collection
of all functions $u(t,x)$ such that the $L_q$-norm
of $\|u(t,\cdot)\|_{L_p(\Omega)}$, as a function of $t \in \bR$,
is finite.

The aim of this paper is to prove the existence 
and uniqueness of solutions
to equations as in \eqref{para_main_eq}
with coefficients satisfying: 
\begin{itemize}
\item[(i)]
$a^{11}$ is measurable in $t \in \bR$ and VMO in $x \in \bR^d$,
\item[(ii)]
$a^{ij}$, $i \ne 1$ or $j \ne 1$, 
are measurable in $(t,x^1) \in \bR^2$ and VMO in $x' \in \bR^{d-1}$.
\end{itemize}
The coefficients $b^{i}(t,x)$ and $c(t,x)$ are assumed to be only measurable and bounded. 
Under these assumptions,
for $f \in L_{q,p}((0,T) \times \Omega)$, $q \ge p \ge 2$,
we find a unique solution $u \in W_{q,p}^{1,2}((0,T) \times \Omega)$,
$u(T,x) = 0$, to the equation \eqref{para_main_eq}.
We also investigate the case $1<q \le p \le 2$
under additional assumptions on $a^{ij}$ (see assumptions before
Theorem \ref{theorem05212007}).

Note that $a^{ij}$, $i \ne 1$ or $j \ne 1$,
are only measurable (i.e., no regularity assumptions) in $x^1$,
so one can say that the class of coefficients considered in this paper 
is strictly bigger than those previously investigated,
for example, in \cite{MR1239929, MR1935919, MR2286441, Krylov_2005, Krylov_2007_mixed_VMO},
where {\em not necessarily continuous} coefficients are considered.
More precisely, the coefficients $a^{ij}$ in \cite{MR1239929} are VMO as functions of $(t,x) \in \bR^{d+1}$ (i.e. VMO in $(t,x)$).
Coefficients as functions of only $t \in \bR$ are dealt with in \cite{MR1935919}
and parabolic systems with VMO coefficients independent of time are investigated in \cite{MR2286441}.
The class of coefficients $a^{ij}$ measurable in time and VMO in the spatial variables (namely, $VMO_{x}$ coefficients) was first introduced in \cite{Krylov_2005}.
Later, the same class of coefficients was investigated in spaces with mixed norms in \cite{Krylov_2007_mixed_VMO}.

In addition to the fact that more general coefficients are available in the $L_p$-theory of parabolic equations,
another benefit of having coefficients measurable in one spatial variable
is that one can deal with parabolic equations in a half-space by only using the solvability of equations in the whole space, $\bR^{d+1}$ or $(S,T) \times \bR^d$.
Roughly speaking, one extends a given equation
defined in a half-space to the whole space using an odd or even extension,
and finds a unique solution to the extended equation
in the whole space.
Then the solution (to the extended equation) gives a unique solution to the original equation.
As is seen in the proof of Theorem \ref{theorem05132007},
an extension of an equation to the whole space
requires, in particular, the odd extensions of the coefficients $a^{1j}$,
$j = 2, \cdots, d$.
Even if $a^{1j}(t,x)$ are constant,
the odd extensions of $a^{1j}(t,x)$ are not continuous or not even in the space of VMO as functions in the whole space.
Thus if we were to consider equations with only VMO (or $VMO_x$) coefficients,
it wouldn't be possible to solve the extended equation in the whole space.
However, due to the solvability of equations in the whole space 
with coefficients $a^{ij}$,
$i \ne 1$ or $j \ne 1$, measurable in $x^1 \in \bR$ as well as in $t \in \bR$,
the extended equation has a unique solution.
This way of dealing with equations in a half-space removes
the necessity of boundary $L_p$-estimates
for solutions to equations in a half-space (or in a bounded domain).
For instance, in \cite{MR1239929} boundary estimates are obtained to have $L_p$-estimates for equations in a bounded domain.

The results for equations in a half-space together with a partition of unity
allow us to solve equations in a bounded domain,
so our results for equations in a half-space with Dirichlet or oblique derivative conditions
can be used to deal with equations with 
$VMO_{x}$ coefficients in a bounded domain.
To the best of our knowledge,
no literature is available for parabolic equations with $VMO_{x}$ coefficients in a bounded domain.
On the other hand, the results in this paper for equations in a half-space  provide a generalization of Corollary 1.3 in 
\cite{MR1828321}, where $a^{ij}$ are measurable functions of only $t \in \bR$, but $a^{1j}$, $j = 2, \cdots, d$, are assumed to be zero.

Slightly different classes of coefficients 
for parabolic equations are considered in 
\cite{Doyoon&Krylov:parabolic:2006, Doyoon:parabolic:2006, Doyoon:parabolic:mixed:2006}.
Especially, the paper \cite{Doyoon:parabolic:mixed:2006} and this paper have almost the same type of methods and results.
However, the main difference is that the coefficient $a^{11}$ in this paper is  measurable in $t$ and VMO in $x \in \bR^d$,
whereas the coefficient $a^{11}$ in \cite{Doyoon:parabolic:mixed:2006}
is measurable in $x^1 \in \bR$ and VMO in $(t,x') \in \bR \times \bR^{d-1}$.
One advantage of $a^{11}$ being as in this paper is that the even extension of $a^{11}$ is again VMO in $x \in \bR^d$ and measurable in $t \in \bR$.
This, indeed, allows us to deal with parabolic equations with coefficients measurable in $t \in \bR$ in a half-space or in a bounded domain.

For more references about elliptic or parabolic equations in Sobolev spaces
with or without mixed norms,
see \cite{
MR1191890, 
MR1088476, 
MR1239929, 
MR2260015, 
MR1941247, 
MR2106068, 
MR2153140, 
MR2110431, 
MR2110205, 
MR2286441, 
MR1935919, 
Krylov_2007_mixed_VMO,
Doyoon:article_para:2004, 
Doyoon:article_half_cont:2004, 
Doyoon&Krylov:article_meas:2005
}
and references therein.

The organization of this paper is as follows.
In section \ref{section05132007} we state the main results of this paper.
The first main result is proved in section \ref{section05202007} and 
the other results are proved using the first main result.
In section \ref{section04092007}
we treat parabolic equations in $L_p$.
Finally, we prove the first main result in section \ref{section05202007}.

A few words about notation:
$(t,x) = (t,x^1,x') \in \bR \times \bR^{d} = \bR^{d+1}$,
where $t \in \bR$, $x^1 \in \bR$, $x' \in \bR^{d-1}$,
and $x = (x^1,x') \in \bR^d$.
By $u_{x'}$ we mean one of $u_{x^j}$, $i = 2, \cdots,d$, or the whole collection
$\{ u_{x^2}, \cdots, u_{x^d} \}$.
As usual, $u_{x}$ represents one of $u_{x^i}$, $i = 1, \cdots, d$,
or the whole collection of $\{ u_{x^1}, \cdots, u_{x^d} \}$.
Thus $u_{xx'}$ is one of $u_{x^ix^j}$,
where $i \in \{1, \cdots, d\}$ and $j \in \{2, \cdots, d\}$,
or the collection of them.
The average of $u$ over an open set $\cD \subset \bR^{d+1}$
is denoted by $\left(u\right)_{\cD}$, i.e.,
$$
\left(u\right)_{\cD}
= \frac{1}{|\cD|} \int_{\cD} u(t,x) \, dx \, dt
= \dashint_{\cD} u(t,x) \, dx \, dt,
$$
where $|\cD|$ is the $d+1$-dimensional volume of $\cD$.
Finally, various constants are denoted by $N$, 
their values may change from one place to another. 
We write $N(d, \delta, \dots)$ if $N$ depends only on $d$, $\delta$, $\dots$.

\noindent{\sc Acknowledgement}: I would like to thank Hongjie Dong 
for his helpful discussions.

\section{Main results}							\label{section05132007}

The coefficients of the parabolic equation \eqref{para_main_eq} 
satisfy the following assumption.

\begin{assumption}\label{assum_01}
The coefficients $a^{ij}$, $b^{i}$, and $c$ are measurable functions defined on $\bR^{d+1}$, $a^{ij} = a^{ji}$.
There exist positive constants $\delta \in (0,1)$ and $K$ such that
$$
|b^{i}(t, x)| \le K, \quad |c(t, x)| \le K,
$$
$$
\delta |\vartheta|^2 \le \sum_{i,j=1}^{d} a^{ij}(t,x) \vartheta^i \vartheta^j
\le \delta^{-1} |\vartheta|^2
$$
for any $(t,x) \in \bR^{d+1}$ and $\vartheta \in \bR^d$.
\end{assumption}

In addition to this assumption, as discussed in the introduction,
we have another assumption on the coefficients $a^{ij}$.
We state this assumption using the following notation.
Let 
$$
B_r(x) = \{ y \in \bR^d: |x- y| < r \},
\quad
Q_r(t,x) = (t, t+r^2) \times B_r(x),
$$
$$
B'_r(x') = \{ y' \in \bR^{d-1}: |x'- y'| < r \},
$$
$$
\Lambda_{r}(t,x)= (t, t+r^2) \times (x^1-r, x^1+r) \times B'_r(x').
$$
Set $B_r = B_r(0)$, $B'_r = B'_r(0)$, $Q_r = Q_r(0)$ and so on. 
By $|B'_r|$ we mean  the $d-1$-dimensional volume of $B'_r(0)$.
Denote 
$$
\text{osc}_{x'} \left( a^{ij}, \Lambda_r(t,x) \right) 
= r^{-3} |B'_r|^{-2} \int_{t}^{t+r^2}\int_{x^1-r}^{x^1+r} A^{ij}_{x'}(s,\tau)  \, d\tau \,  ds, 
$$
$$
\text{osc}_{x} \left( a^{ij}, Q_r(t,x) \right) 
= r^{-2} |B_r|^{-2}  \int_{t}^{t+r^2}  A^{ij}_{x}(\tau) \, d\tau, 
$$
where
$$
A^{ij}_{x'}(s,\tau)
= \int_{y', z' \in B'_r(x')}
|a^{ij}(s, \tau, y') - a^{ij}(s, \tau, z') | \, dy' \, dz',
$$
$$
A^{ij}_{x}(\tau)
= \int_{ y, z \in B_r(x)}
|a^{ij}(s, y) - a^{ij}(s, z) | \, dy \, dz.
$$
Also denote
$$
\cO_R^{\, x'}(a^{ij}) = \sup_{(t,x) \in \bR^{d+1}} \sup_{r \le R} \,\,\, \text{osc}_{x'} \left( a^{ij}, \Lambda_r(t,x) \right),
$$
$$
\cO_R^{\, x}(a^{ij}) = \sup_{(t,x) \in \bR^{d+1}} \sup_{r \le R} \,\,\, \text{osc}_{x} \left( a^{ij}, B_r(t,x) \right).
$$
Finally set
$$
a_R^{\#} =\cO_R^{\, x}(a^{11}) + \sum_{i \ne 1 \, \text{or} \, j \ne 1} \cO_R^{\, x'}(a^{ij}).
$$

\begin{assumption}\label{assum_02}
There is a continuous function $\omega(t)$ defined on $[0,\infty)$ such that $\omega(0) = 0$ and $a_R^{\#} \le \omega(R)$ for all $R \in [0,\infty)$.
\end{assumption}

Let $\Omega$ be either $\bR^d$ or $\bR^d_+$.
We consider the space $W_{q,p}^{1,2}((S,T) \times \Omega)$,
$-\infty \le S < T \le \infty$,
which is the collection of all functions defined on $(S,T) \times \Omega$
such that 
$$
\| u \|_{W_{q,p}^{1,2}((S,T) \times \Omega)}
:= \|u \|_{L_{q,p}((S,T) \times \Omega)} 
+ \| u_{x} \|_{L_{q,p}((S,T) \times \Omega)} 
$$
$$
+ \| u_{x x} \|_{L_{q,p}((S,T) \times \Omega)}
+ \| u_t \|_{L_{q,p}((S,T) \times \Omega)} < \infty.
$$
By $u \in \WO{1,2}{q,p}((S,T) \times \bR^d)$
we mean $u \in  W_{q,p}^{1,2}((S,T) \times \bR^d)$
and $u(T,x) = 0$.
Throughout the paper, we set
$$
L_{q,p} := L_{q,p} (\bR \times \bR^d),
\quad
W_{q,p}^{1,2} : = W_{q,p}^{1,2}(\bR \times \bR^{d}).
$$
In case $p = q$, we have
$$
L_p((S, T) \times \Omega)
= L_{p,p}((S, T) \times \Omega),
$$
$$
W_p^{1,2}((S,T) \times \Omega)
= W_{p,p}^{1,2}((S,T) \times \Omega).
$$

We denote the differential operator by $L$, that is,
$$
Lu = u_t + a^{ij} u_{x^i x^j} + b^{i} u_{x^i} + c u.
$$
The following are the main results of this paper.

\begin{theorem}							\label{theorem03192007}
Let $q \ge p \ge 2$, $0 < T < \infty$,
and the coefficients of $L$ satisfy
Assumption \ref{assum_01} and \ref{assum_02}.
In addition, if $p = 2$, the coefficients of $L$
are assumed to be independent of $x' \in \bR^{d-1}$.
Then for any $f \in L_{q,p}((0,T) \times \bR^d)$, 
there exists a unique $u \in \WO{1,2}{q,p}((0,T) \times \bR^d)$ such that 
$Lu = f$ in $(0,T) \times \bR^d$.
Furthermore, there is a constant $N$, depending only on $d$, $p$, $q$, $\delta$, $K$, $T$, and $\omega$,
such that, for any $u \in \WO{1,2}{q,p}((0,T) \times \bR^d)$,
$$
\| u \|_{W_{q,p}^{1,2}((0,T) \times \bR^d)}
\le N \| Lu \|_{L_{q,p}((0,T) \times \bR^d)}.
$$
\end{theorem}

\begin{remark}							\label{remark05212007}
In the above theorem, if $p = q = 2$, by Theorem 2.2 in \cite{Doyoon&Krylov:parabolic:2006}
the coefficients $a^{ij}(t,x)$ are allowed to be measurable functions of $(t,x^1) \in \bR^2$ including $a^{11}$.
The same argument applies to Theorems \ref{theorem05212007} and
\ref{theorem05132007} below.
On the other hand, whenever we have coefficients $a^{ij}$ independent 
of $x'' \in \bR^{m}$, $m \le d$, we can replace them by coefficients $a^{ij}(t,x)$ which are uniformly continuous with respect to $x''$ uniformly in the remaining variables.
\end{remark}

The next theorem considers the case with $1 < q \le p \le 2$.
In this case, we assume that the coefficients $a^{ij}$ of L satisfy one of the following assumptions (recall that $a^{ij} = a^{ji}$):
\begin{itemize}

\item[(i)]
The coefficients $a^{1j}$, $j=2,\cdots, d$, are measurable functions of $(t,x^1) \in \bR^2$ and the other coefficients $a^{ij}$ are functions of only $t \in \bR$. That is,
\begin{equation}							\label{05252007_01}
\left\{
\begin{aligned}
a^{ij}(t,x) &= a^{ij}(t), \quad i=j=1 \quad \text{or} \quad 
i,j \in \{2,\cdots,d\}\\
a^{1j}(t,x) &= a^{1j}(t,x^1), \quad j = 2, \cdots, d
\end{aligned}\right..
\end{equation}

\item[(ii)]
The coefficients $a^{ij}$, $i,j \ge 2$, are measurable functions of $(t,x^1) \in \bR^2$ and the other coefficients $a^{ij}$ are functions of only $t \in \bR$.
That is,
\begin{equation}							\label{05252007_02}
\left\{
\begin{aligned}
a^{1j}(t,x) &= a^{1j}(t), \quad j = 1, \cdots, d\\
a^{ij}(t,x) &= a^{ij}(t,x^1), \quad i,j \in \{2,\cdots,d\}
\end{aligned}\right..
\end{equation}

\end{itemize}

\begin{theorem}							\label{theorem05212007}
Let $1 < q \le p \le 2$
and the coefficients $a^{ij}$ of $L$ be as above.
Then for any $f \in L_{q,p}((0,T) \times \bR^d)$, 
there exists a unique $u \in \WO{1,2}{q,p}((0,T) \times \bR^d)$ such that 
$Lu = f$ in $(0,T) \times \bR^d$.
Furthermore, there is a constant $N$, depending only on $d$, $p$, $q$, $\delta$, $K$, and $T$,
such that
\begin{equation}							\label{05252007_04}
\| u \|_{W_{q,p}^{1,2}((0,T) \times \bR^d)}
\le N \| Lu \|_{L_{q,p}((0,T) \times \bR^d)}
\end{equation}
for any $u \in \WO{1,2}{q,p}((0,T) \times \bR^d)$.
\end{theorem}

\begin{proof}
Without loss of generality, we assume that $b^i = c = 0$.
Moreover, it is enough to prove the estimate in the theorem.
Let $u$ be such that 
$u \in W_{q,p}^{1,2}((0,T) \times \bR^d)$ and $u(T,x) = 0$.

\noindent{\bf Case 1}.
Let the coefficients $a^{ij}$ of $L$ satisfy the assumption \eqref{05252007_01}.
For $\phi \in C_0^{\infty}((0,T) \times \bR^d)$,
find $v \in W_{q',p'}^{1,2}((0,T) \times \bR^d)$, $q' = q/(q-1)$,
$p' = p/(p-1)$ such that $v(0,x) = 0$ and
$$
-v_t + a^{ij}(t,x)v_{x^ix^j} = \phi.
$$
This is possible due to Theorem \ref{theorem03192007}
along with the fact that $2 \le p' \le q'$.
Observe that 
\begin{multline}							\label{05252007_03}
\int_{(0,T) \times \bR^d}
u_{x^1x^k} \phi \, dx \, dt
= \int_{(0,T) \times \bR^d} u_{x^1x^k} \left(-v_t + a^{ij}(t,x)v_{x^ix^j}\right) \, dx \, dt
\\
= \int_{(0,T) \times \bR^d} \left(u_t + a^{ij}(t,x) u_{x^ix^j}\right)v_{x^1x^k} \, dx \, dt
\end{multline}
for $k = 2, \cdots, d$.
Indeed, the second equality above is obtained using the fact that
$a^{ij}(t,x)$ are independent of $x \in \bR^d$ if $i=j=1$
or $i,j \in \{2,\cdots,d\}$
and $a^{1j}(t,x) = a^{1j}(t,x^1)$ if $j = 2,\cdots,d$.
Especially,
$$
\int_{(0,T) \times \bR^d} u_{x^1x^k} a^{1j}(t,x) v_{x^1x^j} \, dx \, dt
=\int_{(0,T) \times \bR^d} u_{x^1x^k} a^{1j}(t,x^1) v_{x^1x^j} \, dx \, dt
$$
$$
= \int_{(0,T) \times \bR^d} u_{x^1x^j} a^{1j}(t,x) v_{x^1x^k} \, dx \, dt,
\quad j,k = 2,\cdots,d.
$$
Therefore, we have
$$
\int_{(0,T) \times \bR^d}
u_{x^1x^k} \phi \, dx \, dt
\le \| Lu \|_{L_{q,p}((0,T) \times \bR^d)} 
\|v_{xx}\|_{L_{q',p'}((0,T) \times \bR^d)}
$$
$$
\le 
N \| Lu \|_{L_{q,p}((0,T) \times \bR^d)}
\| \phi \|_{L_{q',p'}((0,T) \times \bR^d)}.
$$
where the last inequality is due to Theorem \ref{theorem03192007}.
This implies that
\begin{equation}							\label{05182007_01}
\| u_{x^1x^k} \|_{L_{q,p}((0,T) \times \bR^d)}
\le N \| Lu \|_{L_{q,p}((0,T) \times \bR^d)},
\quad
k = 2, \cdots, d.
\end{equation}

Now we set 
$$
L_1 u := u_t + a^{ij}(t) u_{x^ix^j},
$$
where $a^{ij}(t) = a^{ij}(t,0)$.
Note that $a^{ij}(t)$ are independent of $x \in \bR^d$,
thus by results in \cite{MR1935919} or \cite{MR1828321}
we have
\begin{equation}							\label{05182007_02}
\| u \|_{W_{q,p}^{1,2}((0,T) \times \bR^d)}
\le N \| L_1u \|_{L_{q,p}((0,T) \times \bR^d)}.
\end{equation}
We see that
$$
L_1 u = L u + 2 \sum_{j=2}^{d} \left( a^{1j}(t) - a^{1j}(t,x^1) \right) u_{x^1x^j}.
$$
This along with \eqref{05182007_01} and \eqref{05182007_02} implies the estimate \eqref{05252007_04}.

\noindent{\bf Case 2}. Now assume that $a^{ij}$ satisfy the assumption \eqref{05252007_02}.
In this case,
since $a^{1j}$, $j=1,\cdots,d$ are independent of $x \in \bR^d$
and $a^{ij}$, $i,j \ge 2$, are independent of $x' \in \bR^{d-1}$,
we see that the integrations by parts in \eqref{05252007_03} 
are possible for $u_{x^kx^l}$, $k, l = 2, \cdots, d$.
Thus we have estimates as in \eqref{05182007_01}
for $u_{x^kx^l}$, $k, l = 2, \cdots, d$.
Then the proof can be completed by
repeating the argument using $L_1$ as above.
Especially, we see 
$$
L_1 u = L u + \sum_{i,j=2}^{d} \left( a^{ij}(t) - a^{ij}(t,x^1) \right) u_{x^ix^j}.
$$
The theorem is proved.
\end{proof}

Next two theorems concern Dirichlet or oblique derivative problems for parabolic equations defined in a half-space. 
Depending on the range of $q$ and $p$, we consider the following coefficients $a^{ij}(t,x)$ of the operator $L$:
\begin{itemize}
\item[(i)]
If $q \ge p \ge 2$, the coefficients
$a^{ij}(t,x)$ satisfy Assumption \ref{assum_01} and \ref{assum_02}.
In addition, if $p = 2$,
the coefficients are independent of $x' \in \bR^{d-1}$.
Especially, $a^{11}(t,x^1)$ is measurable in $t$ and VMO in $x^1 \in \bR$
if $p = 2$.
\item[(ii)]
If $1<q \le p \le 2$,
the coefficients $a^{ij}(t,x)$ are measurable functions of only $t \in \bR$ satisfying Assumption \ref{assum_01}.
\end{itemize}

\begin{remark}
More precisely, in case $1<q \le p \le 2$,
the coefficients $a^{1j}$, $j = 2, \cdots, d$ are allowed to be measurable functions of $(t,x^1) \in \bR^2$. 
Moreover, if $a^{1j} = 0$, $j = 2, \cdots, d$, 
then the coefficients $a^{ij}$, $i,j \ge 2$, can be measurable functions of $(t,x^1) \in \bR^2$.
See the proof of the following theorem as well as Theorem \ref{theorem05212007}.
\end{remark}

\begin{theorem}							\label{theorem05132007}
Let $0 < T < \infty$.
Assume that either we have $1<q \le p \le 2$ 
or $2 \le p \le q$.
Then for any $f \in L_{q,p}((0,T) \times \bR^d_{+})$, 
there exists a unique $u \in W_{q,p}^{1,2}((0,T) \times \bR^d_+)$ such that 
$u(T,x) = u(t,0,x') = 0$ and $Lu = f$ in $(0,T) \times \bR^d_+$.
\end{theorem}

\begin{proof}
Introduce a new operator $\hat{L}v = \hat{a}^{ij} v_{x^i x^j} + \hat{b} v_{x^i} + \hat{c} v$, where $\hat{a}^{ij}$, $\hat{b}^{i}$, and $\hat{c}$ are defined as either even or odd extensions of $a^{ij}$, $b^{j}$, and $c$.
Specifically, 
for $i = j = 1$ and $i,j \in \{ 2, \dots, d \}$, even extensions:
$$
\hat{a}^{ij} = a^{ij}(t,x^1,x') 
\quad x^1 \ge 0,
\qquad
\hat{a}^{ij} = a^{ij}(t,-x^1,x')
\quad x^1 < 0.
$$
For $j = 2, \dots, d$, odd extensions:
$$
\hat{a}^{1j} = a^{1j}(t,x^1,x') 
\quad x^1 \ge 0,
\qquad
\hat{a}^{1j}= - a^{1j}(t,-x^1,x')
\quad x^1 < 0.
$$ Also set $\hat{a}^{j1} = \hat{a}^{1j}$. Similarly, $\hat{b}^1$ is the
odd extension of $b^1$, and $\hat{b}^{i}$, $i = 2, \dots, d$, and
$\hat{c}$ are even extensions of $b^{i}$ and $c$ respectively. We see
that the coefficients $\hat{a}^{ij}$, $\hat{b}^{i}$, and $\hat{c}$
satisfy Assumption \ref{assum_01}.  In addition, if $q \ge p \ge 2$, the
coefficients $\hat{a}^{ij}$ satisfy Assumption \ref{assum_02} 
with $N\omega(3 t)$, where $N$ depends only on $d$. 
Especially, $\hat{a}^{11}$ is VMO in $x \in \bR^d$.

For $f \in L_p((0,T) \times \bR^d_+)$, set $\hat{f}$ to be the odd extension of $f$.  Then it follows from Theorem
\ref{theorem03192007} or Theorem \ref{theorem05212007}
that there exists a unique $u
\in \WO{1,2}{q,p}((0,T) \times \bR^d)$ to the equation $\hat{L}u = \hat{f}$. 
It is easy to
check that $-u(t,-x^1,x') \in \WO{1,2}{q,p}((0,T) \times \bR^d)$ also satisfies the same equation, so by uniqueness we have $u(t,x^1,x') = -u(t,-x^1,x')$.
This and the fact that $u \in \WO{1,2}{q,p}((0,T) \times \bR^d)$ show that $u$,
as a function defined on $(0,T) \times \bR^{d}_{+}$, is a solution to $Lu
= f$ satisfying $u = 0$ on $\{(T,x):x\in\bR^d\}$ and $\{(t,0,x'): 0 \le t \le T, x' \in \bR^{d-1}\}$.

Uniqueness follows from the fact that the odd extension of a solution $u$
belongs to $\WO{1,2}{q,p}((0,T) \times \bR^d)$ and the uniqueness of solutions to
equations in $(0,T) \times \bR^d$.
\end{proof}

This theorem addresses the oblique derivative problem.

\begin{theorem}							\label{theorem05172007}
Let $p$, $q$, and 
$a^{ij}$ be as in Theorem \ref{theorem05132007}.
Let $\ell = (\ell^1, \cdots, \ell^d)$ be a vector in $\bR^d$ with $\ell^1 > 0$.
Then for any $f \in L_{q,p}((0,T) \times \bR^d_+)$, there exists a unique $u \in W_{q,p}^{1,2}((0,T) \times \bR^d_+)$ satisfying
$Lu = f$ in $(0,T) \times \bR^{d}_{+}$,
$\ell^j u_{x^j} = 0$ on $\{(t,0,x'): 0 \le t \le T, x' \in \bR^{d-1}\}$,
and 
$u = 0$ on $\{(T,x): x \in \bR^d\}$.
\end{theorem}

\begin{proof}
Let
$\varphi(x) = (\ell^1 x^1, \ell'x^1 + x')$, 
where $\ell' = (\ell^2, \dots, \ell^d)$.
Using this linear transformation and its inverse,
we reduce the above problem to a problem with zero Neumann boundary condition on 
$\{(t,0,x'): 0 \le t \le T, x' \in \bR^{d-1}\}$.
Note that, in case $q \ge p \ge 2$, 
the coefficients of the transformed equation satisfy Assumption \ref{assum_02} with $N \omega(N t)$, where $N$ depends only on $d$ and $\ell$.
Then the problem is solved as in the proof of Theorem \ref{theorem05132007} with the even extension of $f$.
\end{proof}

\begin{remark}
Appropriate $L_{q,p}$-estimates as in Theorem \ref{theorem03192007}
can be added to the above two theorems.
\end{remark}

\section{Parabolic equations in $L_p$} \label{section04092007}

In this section we prove Theorem \ref{theorem03192007} for the case $p = q > 2$.
In fact, we prove Theorem \ref{theorem05192007_01} below,
which implies Theorem \ref{theorem03192007} if $p = q > 2$.
As in Theorem \ref{theorem03192007},
we assume that the coefficients $a^{ij}$, $b^i$, and $c$ of $L$
satisfy Assumption \ref{assum_01} and \ref{assum_02}.

\begin{theorem}							\label{theorem05192007_01}
Let $p > 2$, $T \in [-\infty, \infty)$, and the coefficients of $L$ satisfy
Assumption \ref{assum_01} and \ref{assum_02}.
Then there exist constants $\lambda_0$ and $N$,
depending only on $d$, $p$, $\delta$, $K$, and the function $\omega$,
such that,
for any $\lambda \ge \lambda_0$ and $u \in W_p^{1,2}((T,\infty)\times\bR^d)$,
$$
\|u_t\|_{L_p((T,\infty)\times\bR^d)} + \|u_{xx}\|_{L_p((T,\infty)\times\bR^d)}
+ \sqrt{\lambda} \|u_x\|_{L_p((T,\infty)\times\bR^d)}
$$
$$
+ \lambda \| u \|_{L_p((T,\infty)\times\bR^d)}
\le N \|L u - \lambda u\|_{L_p((T,\infty)\times\bR^d)}.
$$
Moreover,
for any $\lambda > \lambda_0$
and $f \in L_p((T,\infty)\times\bR^d)$,
there exists a unique solution $u \in W_p^{1,2}((T,\infty)\times\bR^d)$
to the equation $Lu - \lambda u = f$.
\end{theorem}

A proof of this theorem is given at the end of this section after a sequence of auxiliary results.
The first result is a lemma which deals with an operator whose coefficients are measurable functions of only $(t,x^1) \in \bR^2$ (except $a^{11}$).
Set
$$
\bar{L}_0 u = u_t + \bar{a}^{ij}(t,x^1) u_{x^i x^j},
$$
where $\bar{a}^{11}(t)$ is a function of only $t \in \bR$
and $\bar{a}^{ij}$, $i \ne 1$ or $j \ne 1$, 
are functions of $(t,x^1) \in \bR^2$.
The coefficients $\bar{a}^{ij}$ satisfy Assumption \ref{assum_01}.

\begin{lemma} \label{lemma03282007_100}
Let $p \ge 2$.
There is a constant $N$, depending only on $d$, $p$, and $\delta$,
such that, for any $u \in W_p^{1,2}(\bR^{d+1})$,
$r \in (0, \infty)$, and $\kappa \ge 8/\delta$,
$$
\dashint_{Q_r} | u_{xx'}(t,x) - \left( u_{xx'} \right)_{Q_r} |^p
\, dx \, dt
\le N \kappa^{d+2} \left( |\bar{L}_0 u|^p \right)_{Q_{\kappa r}}
+ N \kappa^{-\nu p} \left( |u_{xx}|^p \right)_{Q_{\kappa r}},
$$
where $\nu = 1/2 - 3/(4p)$.
\end{lemma}

\begin{proof}
It can be said that the lemma is proved by following the arguments in section 5 of the paper \cite{Doyoon:parabolic:mixed:2006}.
In fact, the above lemma would be the same as Theorem 5.9 in \cite{Doyoon:parabolic:mixed:2006}
if the coefficient $\bar{a}^{11}$ were a function of only $x^1 \in \bR$.
In our case, the coefficient $\bar{a}^{11}$ is a function of only $t \in \bR$.
Thus, instead of repeating the steps in \cite{Doyoon:parabolic:mixed:2006}
for the operator $\bar{L}_0$,
one can use a time change as well as Theorem 5.9 in \cite{Doyoon:parabolic:mixed:2006}.
Indeed, we can proceed as follows.

Without loss of generality we assume that
$\bar{a}^{ij}(t,x^1)$ are infinitely differentiable as functions of $t \in \bR$.
Especially, we may assume that the derivative of $\bar{a}^{11}(t)$ is bounded.
For example, we can consider
$$
\bar{a}^{ij}_{\varepsilon}(t,x^1) = \int_{\bR} \bar{a}^{ij}(s,x^1) \phi_{\varepsilon}(t-s) \, ds,
$$
where $\phi \in C_0^{\infty}(\bR)$ such that $\|\phi \|_{L_1(\bR)} = 1$.
Clearly the derivative of $\bar{a}^{11}_{\varepsilon}(t)$ is bounded by a constant depending on $\varepsilon$, but it will be seen that the constant $N$ in the desired estimate does not depend on $\varepsilon$.
Then we let $\varepsilon \searrow 0$.

The additional condition on $\bar{a}^{11}(t)$ assures that
there exists $\varphi(t)$ such that
$$
\varphi(t) = \int_0^t \frac{1}{\bar{a}^{11}\left(\varphi(s)\right)} \, ds.
$$
There also exists $\eta(t)$, the inverse function of $\varphi(t)$.
For $u \in W_p^{1,2}(\bR^{d+1})$, set
$w(t,x) = u(\varphi(t),x)$ and
$$
\fL w := w_t + \hat{a}^{ij}(t,x^1) w_{x^ix^j},
\quad
\hat{a}^{ij}(t,x^1) := \frac{\bar{a}^{ij}(\varphi(t),	x^1)}{\bar{a}^{11}(\varphi(t))}.
$$
Observe that $\hat{a}^{ij}$ are measurable functions of $(t,x^1) \in \bR^2$ satisfying Assumption \ref{assum_01} with $\delta^2$ in stead of $\delta$. 
Moreover, $\hat{a}^{11} = 1$.
Thus by Theorem 5.9 in \cite{Doyoon:parabolic:mixed:2006}
we have 
$$
\dashint_{Q_r} | w_{xx'}(t,x) - c |^p
\, dx \, dt
\le N \kappa^{d+2} \left( |\fL w|^p \right)_{Q_{\kappa r}}
+ N \kappa^{-\nu p} \left( |w_{xx}|^p \right)_{Q_{\kappa r}}
$$
for $r \in (0, \infty)$ and $\kappa \ge 8$,
where $c = \left(w_{xx'}\right)_{Q_r}$
and $N$ depends only on $d$, $p$, and $\delta$.
Using this inequality as well as an appropriate change of variable ($w(t,x) = u(\varphi(t),x)$),
we obtain
$$
r^{-2}\int_0^{\varphi(r^2)} \!\!\!\! \dashint_{B_r} | u_{xx'}(t,x) - c |^p
\, dx \, dt
\le
N (\kappa r)^{-2} \kappa^{d+2} \int_0^{\varphi((\kappa r)^2)} \!\!\!\! \dashint_{B_{\kappa r}}
| \bar{L}_0 u |^p \, dx \, dt
$$
$$
+ N (\kappa r)^{-2} \kappa^{-\nu p}
\int_0^{\varphi((\kappa r)^2)} \!\!\!\! \dashint_{B_{\kappa r}}
| u_{xx} |^p \, dx \, dt
$$
for $r \in (0, \infty)$ and $\kappa \ge 8$,
where $N = N(d, p, \delta)$.
From this inequality along with the facts that $\delta \in (0,1)$ and
$\delta t \le \varphi(t) \le \delta^{-1} t$,
it follows that 
$$
\dashint_{Q_{r\sqrt{\delta}}} | u_{xx'}(t,x) - c |^p \, dx \, dt
\le N \kappa^{d+2} \left( |\bar{L}_0 u|^p \right)_{Q_{\kappa r/\sqrt{\delta}}}
+ N \kappa^{-\nu p} \left( |u_{xx}|^p \right)_{Q_{\kappa r/\sqrt{\delta}}},
$$
where $N = N(d,p,\delta)$.
Replace $r \sqrt{\delta}$ with $r$ 
and $\kappa/\delta$ with $\kappa$ in the above inequality
(thus $\kappa \ge 8/\delta$).
Finally, observe that 
$$
\dashint_{Q_r} | u_{xx'}(t,x) - \left( u_{xx'} \right)_{Q_r} |^p \, dx \, dt
\le N(p) \dashint_{Q_r} | u_{xx'}(t,x) - c |^p \, dx \, dt.
$$
The lemma is proved.
\end{proof}

Let $\bQ$  be the collection of all $Q_r(t,x)$, $(t,x) \in \bR^{d+1}$,
$r \in (0, \infty)$.
For a function $g$ defined on $\bR^{d+1}$,
we denote its (parabolic) maximal and sharp function, respectively, by
$$
M g (t,x) = \sup_{(t,x) \in Q}
\dashint_{Q} | g(s,y) | \, dy \, ds,
$$
$$
g^{\#}(t,x) = \sup_{(t,x) \in Q}
\dashint_{Q} | g(s,y) - (g)_Q | \, dy \, ds,
$$
where the supremums are taken over all $Q \in \bQ$ containing $(t,x)$.
By $L_0$ we mean the operator $L$ with $b^i = c = 0$, i.e.,
$$
L_0 u = u_t + a^{ij}(t,x) u_{x^ix^j}.
$$

\begin{theorem} \label{main_sharp}
Let $\mu$, $\nu \in (1,\infty)$, $1/\mu + 1/\nu = 1$, and $R \in (0, \infty)$.
There exists a constant $N = N(d, \delta, \mu)$ such that,
for any $u \in C_0^{\infty}(\bR^{d+1})$ vanishing outside $Q_R$, 
we have
$$
(u_{xx'})^{\#} \le N ( a_R^{\#} )^{\frac{\alpha}{\nu}} 
\left[ M ( |u_{xx}|^{2 \mu} ) \right]^{\frac{1}{2\mu}}
+ N \left[ M ( |L_0 u|^{2} ) \right]^{\alpha}
\left[ M ( |u_{xx}|^{2} ) \right]^{\beta},
$$
where $\alpha = 1/(8d + 18)$ and $\beta = (4d + 8)/(8d + 18)$.
\end{theorem}

\begin{proof}
Let $\kappa \ge 8/\delta$, $r \in (0, \infty)$, and $(t_0, x_0) = (t_0, x^1_0, x_0') \in \bR^{d+1}$.
We introduce another coefficients $\bar{a}^{ij}$ defined as follows.
$$
\bar{a}^{11}(t) = 
\dashint_{B_{\kappa r}(x_0)} a^{11}(t,y) \, dy
\quad \text{if} \quad \kappa r < R,
$$
$$
\bar{a}^{11}(t) = 
\dashint_{B_{R}} a^{11}(t,y) \, dy
\quad \text{if} \quad \kappa r \ge R.
$$
In case $i \ne 1$ or $j \ne 1$,
$$
\bar{a}^{ij}(t,x^1) = 
\dashint_{B'_{\kappa r}(x_0')} a^{ij}(t,x^1,y') \, dy'
\quad \text{if} \quad \kappa r < R,
$$
$$
\bar{a}^{ij}(t,x^1) = 
\dashint_{B'_{R}} a^{ij}(t,x^1,y') \, dy' 
\quad \text{if} \quad \kappa r \ge R.
$$

Set $\bar{L}_0 u = u_t + \bar{a}^{ij} u_{x^i x^j}$.
Then by Lemma \ref{lemma03282007_100} with an appropriate translation, 
we have
\begin{multline}\label{sharp_u}
\left( |u_{xx'} - (u_{xx'})_{Q_r(t_0, x_0)}|^2 \right)_{Q_r(t_0, x_0)} \\
\le N \kappa^{d+2} \left( |\bar{L}_0 u|^2 \right)_{Q_{\kappa r}(t_0, x_0)}
+ N \kappa^{-1/4} \left( |u_{xx}|^2 \right)_{Q_{\kappa r}(t_0, x_0)}.
\end{multline}
Note that
\begin{equation} \label{esti_chi}
\int_{Q_{\kappa r}(t_0, x_0)}
|\bar{L}_0 u|^2 \, dx \, dt
\le 2 \int_{Q_{\kappa r}(t_0, x_0)}
|L_0 u|^2 \, dx \, dt + N(d) \sum_{i,j=1} \chi_{ij},
\end{equation}
where
$$
\chi_{ij} = \int_{Q_{\kappa r}(t_0,x_0)}
| ( \bar{a}^{ij} - a^{ij} ) u_{x^i x^j} |^2 \, dx \, dt
= \int_{Q_{\kappa r}(t_0,x_0) \cap Q_R} \dots
\le I_{ij}^{1/\nu} J_{ij}^{1/\mu},
$$
$$
I_{ij} = \int_{Q_{\kappa r}(t_0, x_0) \cap Q_R}
| \bar{a}^{ij} - a^{ij} |^{2 \nu} \, dx \, dt,
$$
$$
J_{ij} = \int_{Q_{\kappa r}(t_0, x_0) \cap Q_R} 
|u_{x^i x^j}|^{2 \mu} \, dx \, dt.
$$
Using the definitions of $\bar{a}^{ij}$ and assumptions on $a^{ij}$,
we obtain the following estimates for $I_{ij}$.
If $\kappa r < R$,
$$
I_{11}
\le N \int_{t_0}^{t_0 + (\kappa r)^2}\int_{B_{\kappa r}(x_0)}
| \bar{a}^{11} - a^{11} | \, dx \, dt
\le N (\kappa r)^{d+2} \cO^{\, x}_{\kappa r}(a^{11})
$$
$$
\le N (\kappa r)^{d+2} a^{\#}_{R}.
$$
In case $\kappa r \ge R$,
$$
I_{11}
\le N \int_{0}^{R^2}\int_{B_{R}}
| \bar{a}^{11} - a^{11} | \, dx \, dt
\le N R^{d+2} \cO^{\, x}_{R}(a^{11})
$$
$$
\le N (\kappa r)^{d+2} a^{\#}_{R}.
$$
Now let $j \ne 1$ or $k \ne 1$.
If $\kappa r < R$, 
$$
I_{ij}
\le N
\int_{\Lambda_{\kappa r}(t_0, x_0)}
| \bar{a}^{ij} - a^{ij} | \, dx' \, dx^1 \, dt
\le N (\kappa r)^{d+2} \cO^{\, x'}_{\kappa r}(a^{ij})
$$
$$
\le N (\kappa r)^{d+2} a^{\#}_{R}.
$$
In case $\kappa r \ge R$, 
$$
I_{ij}
\le N
\int_{\Lambda_{R}}
| \bar{a}^{ij} - a^{ij} | \, dx' \, dx^1 \, dt
\le N R^{d+2} \cO^{\, x'}_{R}(a^{ij})
$$
$$
\le N (\kappa r)^{d+2} a^{\#}_{R}.
$$
From the inequality \eqref{esti_chi} and the estimates for $I_{ij}$,
it follows that
$$
\left( |\bar{L}_0 u|^2 \right)_{Q_{\kappa r}(t_0, x_0)}
\le N (a^{\#}_R)^{1/\nu} \left( |u_{xx}|^{2 \mu} \right)^{1/\mu}_{Q_{\kappa r}(t_0, x_0)}
+ N \left( |L_0 u|^2 \right)_{Q_{\kappa r}(t_0, x_0)}.
$$
This, together with \eqref{sharp_u}, gives us
\begin{multline}						\label{2006_09_20_01}
\left(|u_{xx'} - (u_{xx'})_{Q_r(t_0, x_0)}|^2 \right)_{Q_r(t_0, x_0)} 
\le N \kappa^{d+2} (a^{\#}_R)^{1/\nu} \left( |u_{xx}|^{2 \mu} \right)^{1/\mu}_{Q_{\kappa r}(t_0, x_0)} \\
+ N \kappa^{d+2} \left( |L_0 u|^2 \right)_{Q_{\kappa r}(t_0, x_0)}
+ N \kappa^{-1/4} \left( |u_{xx}|^2 \right)_{Q_{\kappa r}(t_0, x_0)}
\end{multline}
for any $r > 0$ and $\kappa \ge 8/\delta$. 
Let 
$$
\cA(t,x) = M(|L_0 u|^2)(t, x), \quad 
\cB(t,x) = M(|u_{xx}|^2)(t, x), 
$$
$$
\cC(t,x) = \left( M(|u_{xx}|^{2 \mu})(t, x) \right)^{1/\mu}.
$$
Then
we observe that $\left( |L_0 u|^2 \right)_{Q_{\kappa r}(t_0, x_0)}
\le \cA(t,x)$ for all $(t,x) \in Q_{r}(t_0,x_0)$.
Similar inequalities are obtained for $\cB$ and $\cC$.
From this and \eqref{2006_09_20_01} 
it follows that, 
for any $(t,x) \in \bR^{d+1}$ and $Q \in \bQ$ such that $(t,x) \in Q$,
$$
\left(|u_{xx'} - (u_{xx'})_{Q}|^2 \right)_{Q}
\le N \kappa^{d+2} (a^{\#}_R)^{1/\nu} \cC(t,x)
$$
$$
+ N \kappa^{d+2} \cA(t,x) 
+ N \kappa^{-1/4} \cB(t,x)
$$
for $\kappa \ge 8/\delta$.
Moreover, the above inequality also holds true for $0 < \kappa < 8/\delta$ because
$$
\dashint_{Q} |u_{xx'} - (u_{xx'})_{Q}|^2 \, dx \, dt 
\le \left(|u_{xx'}|^2 \right)_{Q}
\le (8\delta^{-1})^{1/4} \kappa^{-1/4} \cB(t,x)
$$
for any $(t,x) \in Q \in \bQ$.
Therefore, we finally have
$$
\left(|u_{xx'} - (u_{xx'})_{Q}|^2 \right)_{Q} 
\le N \kappa^{d+2} (a^{\#}_R)^{1/\nu} \cC(t,x)
$$
$$
+ N \kappa^{d+2} \cA(t,x)
+ N \kappa^{-1/4} \cB(t,x)
$$
for all $\kappa > 0$, $(t,x) \in \bR^{d+1}$,
and $Q \in \bQ$ such that $(t,x) \in Q$.
Take the supremum of the left-hand side of the above inequality over
all $Q \in \bQ$ containing $(t,x)$,
and then minimize the right-hand side with respect to $\kappa > 0$.
Also observe that 
$$
{\left(|u_{xx'} - (u_{xx'})_{Q}| \right)_{Q}}^2
\le \left(|u_{xx'} - (u_{xx'})_{Q}|^2 \right)_{Q}.
$$
Then we obtain
$$
\left[u^{\#}_{xx'}(t, x)\right]^2
\le N \left[ (a^{\#}_R)^{1/\nu} \cC(t,x) + \cA(t,x) \right]^{\frac{1}{4d+9}}
\left[\cB(t,x)\right]^{\frac{4d+8}{4d+9}},
$$
where $N = N(d,\delta, \mu)$.
Upon noticing $\cB(t,x) \le \cC(t,x)$, we arrive at the inequality in the theorem.
This finishes the proof.~\end{proof}

\begin{corollary}				\label{corollary_2006_09_13}
For $p>2$, there exist constants $R = R(d,\delta, p,\omega)$
and $N = N(d,\delta,p)$ such that,
for any $u \in C_0^{\infty}(\bR^{d+1})$ vanishing outside $Q_R$, 
we have
$$
\| u_{t} \|_{L_p} + \| u_{xx} \|_{L_p} \le N \|L_0 u\|_{L_p}.
$$
\end{corollary}

\begin{proof}
Let $\mu$ be a real number such that $p > 2 \mu > 2$.
Then by applying the Fefferman-Stein theorem on sharp functions,
H\"{o}lder's inequality, and Hardy-Littlewood maximal function theorem
on the inequality in Theorem \ref{main_sharp},
we obtain
\begin{equation}\label{L_p_esti}
\| u_{xx'} \|_{L_p}
\le N (a_R^{\#})^{\frac{\alpha}{\nu}} \| u_{xx} \|_{L_p}
+ N \|L_0 u\|_{L_p}^{2 \alpha} \|u_{xx}\|_{L_p}^{2 \beta},
\end{equation}
where, as noted in Theorem \ref{main_sharp},
$1/\mu + 1/\nu =1$ and $2 \alpha + 2 \beta = 1$.
On the other hand,
let
$$
g = L_0 u + \Delta_{d-1} u 
- \sum_{i \ne 1, j \ne 1} a^{ij} u_{x^i x^j},
$$
where $\Delta_{d-1} u = u_{x^2 x^2} + \dots + u_{x^d x^d}$.
Then 
$$
u_t + a^{11} u_{x^1 x^1} + \Delta_{d-1} u = g.
$$
Note that the coefficients of the operator 
$$
L_1 u = u_t + a^{11}(t,x) u_{x^1 x^1} + \Delta_{d-1} u
$$
satisfy the assumptions in Corollary 3.7 of \cite{Krylov_2005}. 
Thus there exist $R = R(d, \delta, p, \omega)$ 
and $N = N(d, \delta, p)$ such that
$$
\| u_{x^1x^1} \|_{L_p} \le N \| g \|_{L_p}
$$
if $u$ vanishes outside $Q_R$.
This leads to
$$
\|u_{x^1x^1}\|_{L_p}
\le N \left( \| L_0 u \|_{L_p} + \| u_{xx'} \|_{L_p} \right)
$$
for $u \in C_0^{\infty}(\bR^{d+1})$ vanishing outside $Q_R$.
This and \eqref{L_p_esti} allow us to have
$$
\| u_{xx} \|_{L_p} 
\le N \| L_0 u \|_{L_p}
+ N (a_R^{\#})^{\frac{\alpha}{\nu}} \| u_{xx} \|_{L_p}
+ N \|L_0 u\|_{L_p}^{2 \alpha} \|u_{xx}\|_{L_p}^{2 \beta}. 
$$
Take another sufficiently small $R$ (we call it $R$ again)
which is not greater than the $R$ above,
so that it satisfies 
\begin{equation}				\label{2006_09_13_01}
N (a_R^{\#})^{\frac{\alpha}{\nu}} \le 1/2.
\end{equation}
Then we obtain
$$
\frac{1}{2}\| u_{xx} \|_{L_p} 
\le N \| L_0 u \|_{L_p} + N \|L_0 u\|_{L_p}^{2 \alpha} \|u_{xx}\|_{L_p}^{2 \beta},
$$
which implies that 
$$
\| u_{xx} \|_{L_p} 
\le N \| L_0 u \|_{L_p}.
$$
Finally, observe that 
$$
\|u_t\|_{L_p} = \|L_0 u - a^{ij} u_{x^ix^j}\|_{L_p} 
\le \|L_0 u \|_{L_p} + N \|u_{xx}\|_{L_p}.
$$
This finishes the proof.~\end{proof}

\begin{proof}[\textit{\textbf{Proof of Theorem ~\ref{theorem05192007_01}}}]
We have an $L_p$-estimate for functions with small compact support.
Then the rest of the proof can be done by following the argument in \cite{Krylov_2005}.~\end{proof}

\section{Proof of Theorem \ref{theorem03192007}}						\label{section05202007}

As in section \ref{section04092007},
we set
$$
L_0 u = u_t + a^{ij}(t,x) u_{x^i x^j},
$$
where coefficients $a^{ij}$ satisfy Assumption \ref{assum_01} and \ref{assum_02}.

\begin{lemma}							\label{lemma04022007}
Let $q > p \ge 2$, and $r \in (0,1]$.
Assume that $v \in W_{q,\text{loc}}^{1,2}(\bR^{d+1})$
satisfies $L_0 v = 0$ in $Q_{2r}$.
Then
$$
\left( |v_{xx}|^q \right)_{Q_r}^{1/q}
\le
N \left( |v_{xx}|^2 \right)^{1/2}_{Q_{2r}}
\le
N \left( |v_{xx}|^p \right)_{Q_{2r}}^{1/p},
$$
where $N$ depends only on $d$, $q$, $\delta$, and the function $\omega$.
\end{lemma}

\begin{proof}
This lemma is proved in the same way as Corollary 6.4 in \cite{Krylov_2007_mixed_VMO}.
As discussed in the proof of Lemma 4.1 in \cite{Doyoon:parabolic:mixed:2006},
the key step is to have the estimate 
$$
\| u_{xx} \|_{L_p(Q_r)}
\le N
\left( \|L_0 u\|_{L_p(Q_{\kappa r})}
+ r^{-1}\|u_x\|_{L_p(Q_{\kappa r})}
+ r^{-2}\|u\|_{L_p(Q_{\kappa r})} \right)
$$
for $p \in (2, \infty)$ and $u \in W_{p, \text{loc}}^{1,2}(\bR^{d+1})$,
where $r \in (0,1]$, $\kappa \in (1, \infty)$,
and $N$ depends only on $d$, $p$, $\delta$, $\kappa$, and the function $\omega$.
This is obtained using
Theorem \ref{theorem05192007_01} in this paper
and the argument in the proof of Lemma 6.3 of \cite{Krylov_2007_mixed_VMO}.
~\end{proof}

In the following we state without proofs some results which are necessary for the proof of Theorem \ref{theorem03192007}.
They can be proved following the arguments in \cite{Krylov_2007_mixed_VMO}.
Alternatively, one can follow the proofs of the corresponding statements 
(Theorem 6.1, Corollary 6.2, Lemma 6.3, and Corollary 6.4)
in section 6 (also see section 4) of the paper \cite{Doyoon:parabolic:mixed:2006}.
Note that Lemma \ref{lemma04022007} above is needed in the proof of the following theorem.

\begin{theorem}
Let $p \ge 2$.
In case $p = 2$, we assume that the coefficients $a^{ij}(t,x)$ of $L_0$ are independent of $x' \in \bR^{d-1}$.
Then there exists a constant $N$, depending on $d$, $p$, $\delta$, and 
the function $\omega$, such that,
for any $u \in C_0^{\infty}(\bR^{d+1})$, $\kappa \ge 16/\delta$,
and $r \in (0, 1/\kappa]$,
we have
$$
\dashint_{Q_r}
| u_{xx'}(t,x) - \left( u_{xx'} \right)_{Q_r} |^p \, dx \, dt
$$
$$
\le N \kappa^{d+2} \left( |L_0 u|^p \right)_{Q_{\kappa r}}
+ N \left( \kappa^{-\nu p} + \kappa^{d+2} (a_{\kappa r}^{\#})^{1/2} \right) \left( |u_{xx}|^p \right)_{Q_{\kappa r}},
$$
where $\nu = 1/2 - 3/(4p)$
\end{theorem}

As in \cite{Doyoon:parabolic:mixed:2006},
we use the following notations, which are $1$-dimensional versions of the notations introduced in section \ref{section04092007}.
If $g$ is a function defined on $\bR$,
by $(g)_{(\sfa,\sfb)}$ we mean
$$
(g)_{(\sfa, \sfb)} = \dashint_{(\sfa, \sfb)} g(s) \, ds
= (\sfb - \sfa)^{-1} \int_{\sfa}^{\sfb} g(s) \, ds.
$$
The maximal and sharp function of $g$ are defined by
$$
M g (t) = \sup_{t \in (\sfa, \sfb)} \dashint_{(\sfa, \sfb)} |g(s)| \, ds,
$$
$$
g^{\#} (t) = \sup_{t \in (\sfa, \sfb)} \dashint_{(\sfa, \sfb)} | g(s) - (g)_{(\sfa, \sfb)} | \, ds,
$$
where the supremums are taken over all intervals $(\sfa, \sfb)$
containing $t$.

\begin{corollary}							\label{corollary04102007_01}
Let $p \ge 2$.
In case $p = 2$, we assume that the coefficients $a^{ij}(t,x)$ of $L_0$ are independent of $x' \in \bR^{d-1}$.
Then there exists a constant $N$, depending on $d$, $p$, $\delta$, and 
the function $\omega$, such that,
for any $u \in C_0^{\infty}(\bR^{d+1})$, $\kappa \ge 16/\delta$,
and $r \in (0, 1/\kappa]$,
we have
$$
\dashint_{(0,r^2)} \left| \varphi(t) - (\varphi)_{(0,r^2)} \right|^p \, dt
$$
$$
\le N \kappa^{d+2} (\psi^p)_{(0,(\kappa r)^2)}
+ N \left( \kappa^{-\nu p} + \kappa^{d+2} (a_{\kappa r}^{\#})^{1/2} \right) (\zeta^p)_{(0,(\kappa r)^2)},
$$
where $\nu = 1/2 - 3/(4p)$,
$$
\varphi(t) = \| u_{xx'}(t, \cdot) \|_{L_p(\bR^d)},
$$
$$
\zeta(t) = \| u_{xx}(t, \cdot) \|_{L_p(\bR^d)},
\quad
\psi(t) = \| L_{0} u(t, \cdot) \|_{L_p(\bR^d)}.
$$
\end{corollary}

\begin{lemma}							\label{lemma04102007_01}
Let $p \ge 2$.
In case $p = 2$, we assume that the coefficients $a^{ij}(t,x)$ of $L_0$ are independent of $x' \in \bR^{d-1}$.
Let $R \in (0,1]$ and $u$ be a function in $C_0^{\infty}(\bR^{d+1})$
such that $u(t,x) = 0$ for $t \notin (0,R^4)$.
Then
$$
\varphi^{\#}(t_0)
\le N \kappa^{(d+2)/p} \left(M \psi^p (t_0)\right)^{1/p}
$$
$$
+ N \left( (\kappa R)^{2-2/p} + \kappa^{-\nu} + \kappa^{(d+2)/p} \left(\omega(R)\right)^{1/2p}   \right) \left(M \zeta^p (t_0)\right)^{1/p}
$$
for all $\kappa \ge 16/\delta$ and $t_0 \in \bR$,
where $\nu = 1/2 - 3/(4p)$, $N = N(d, p, \delta,\omega)$,
and the functions $\varphi$, $\zeta$, $\psi$
are defined as in Corollary \ref{corollary04102007_01}.
\end{lemma}

The following corollary 
is proved by repeating word for word 
the proof of Corollary 6.4 in \cite{Doyoon:parabolic:mixed:2006},
but we have to use, instead of Corollary 4.5 in \cite{Doyoon:parabolic:mixed:2006}, the corresponding result
in \cite{Krylov_2007_mixed_VMO} (see Lemma 3.4 and its proof there)
since $a^{11}$ is assumed to be measurable in $t \in \bR$ and VMO in $x \in \bR^d$.

\begin{corollary}
Let $q > p \ge 2$.
Assume that, in case $p = 2$,
the coefficients $a^{ij}$ of $L_0$
are independent of $x' \in \bR^{d-1}$.
Then there exists $R = R(d, p, q, \delta, \omega)$
such that,
for any $u \in C_0^{\infty}(\bR^{d+1})$
satisfying $u(t,x) = 0$ for $t \notin (0, R^4)$,
$$
\|u_t\|_{L_{q,p}} + \|u_{xx}\|_{L_{q,p}}
\le N \| L_0 u \|_{L_{q,p}},
$$
where $N = N(d, p, q, \delta, \omega)$.
\end{corollary}

\begin{proof}[\textit{\textbf{Proof of Theorem ~\ref{theorem03192007}}}]
If $p = q \ge 2$, the theorem follows from
Theorem 2.2 in \cite{Doyoon&Krylov:parabolic:2006}
as well as Theorem \ref{theorem05192007_01} in this paper.
To deal with the case with $q > p \ge 2$,
we use the $L_{q,p}$-estimate proved above for
functions with compact support with respect to $t \in \bR$
and follow the proofs in section 3 of the paper \cite{Krylov_2007_mixed_VMO}.
Theorem \ref{theorem03192007} is now proved.
\end{proof}

\bibliographystyle{plain}

\def\cprime{$'$}\def\cprime{$'$} \def\cprime{$'$} \def\cprime{$'$}
  \def\cprime{$'$} \def\cprime{$'$}

\end{document}